\documentclass[11pt]{amsart}

\usepackage[utf8]{inputenc}
\usepackage[colorlinks=true,linkcolor=blue,citecolor=blue]{hyperref}
\usepackage{amsmath,amssymb}
\usepackage{wrapfig}
\usepackage{url}
\usepackage{indentfirst}
\usepackage{mathtools}
\usepackage{graphicx}
\usepackage{stmaryrd}
\usepackage{amsthm}
\usepackage{relsize}
\usepackage{dsfont}
\usepackage{geometry}
\usepackage{setspace}
\title{Studies on the Pea Pattern Sequence}
\author{André Pedroso Kowacs }
\date{2017}

\newtheorem{theorem}{Theorem}
\newtheorem{corollary}{Corollary}
\newtheorem{lemma}{Lemma}

\begin{document}
 
\maketitle
\section{Abstract}
In this paper we formally define the family of sequences known as ``Pea Pattern sequences''. We then analyze its behaviour and conditions for fixed points. The paper ends with a list of fixed points from base $2$ to $10$.
\section{Introduction}
This paper was inspired in the works by John Conway on the sequence [A005150], also known as the ``Look and Say sequence'' (see \cite{OEIS}). The Pea Pattern sequence, though similarly defined, holds very different properties. It was described in \cite{e3} for some of its cases. The existence of fixed points and cycles in this sequence is the main object of study in this paper, which also compiles every fixed point for the sequences in bases 2 to 10.

\section{Preliminaries}
Note that the correspondence between the natural numbers in some base $k\in\mathbb{N}$ and the set of finite words $\Sigma_k^*$ over the alphabet $\Sigma_k = \{0,1,...,k-1\}$ will be used throughout the paper, as no confusion should arise from it.
The concatenation operator shall be denoted by the symbol $||$, although it will generally be omitted.\\\\
Given $x_0 \in \Sigma_k^*$, define the sequence $(x_n)_n$ recursively by 
\[
x_{n+1}=\mathcal{P}_k(x_n)=a_{(n_1,1)}a_{(n_1-1,1)}...a_{(0,1)}b_1a_{(n_2,2)}...a_{(0,2)}b_2...a_{(n_r,r)}...a_{(0,r)}b_r,
\]
where $(a_{(n_j,j)}...a_{(0,j)})_k=|x_n|_{b_j},$ for $1\leq j\leq r$, with $|x_n|_{b_j} \neq 0,\;1\leq j\leq r$ and $b_1>b_2>...>b_r$. Here $(a_{(n_j,j)}...a_{(0,j)})_k$ denotes the natural number given by the string $a_{(n_j,j)}...a_{(0,j)}$ in base $k$ representation.\\
For example, for $x_0 = 123 \in \Sigma_{10}^*$, we have that $x_1 = 131211$, $x_2 = 131241,\, \text{etc}...$
\begin{theorem}
For all $x_0 \in \Sigma_k$, the sequence defined by $x_{n+1} = \mathcal{P}_k(x_n)$ converges to a fixed point or a cycle.\\
\end{theorem}
\begin{proof}
Note that there exists a finite number of words of up to a given length, so it suffices to show that the terms in the sequence are bounded in length. Indeed, first notice that the numbers of letters in $x_{n+1} \in \Sigma_k$ is given by
\[|x_{n}| = \lfloor\log_k{x_{n}}\rfloor+1.\]
But by definition
\begin{align*}
|x_{n+1}|&=|(|x_n|_{k-1})_k \,|+| (k-1) \,|+| \, ... \, |+| \,(|x_n|_0)_k \, |+| \,0|\\
&\leq|(|x_n|_{k-1})_k \,|+|\, ... \, |+|\,  (|x_n|_0)_k| + k\\
&\leq {k|({|x_n|})_k|}+k = k\lfloor\log_k{{|x|}}\rfloor+2k,
\end{align*}
therefore
\begin{equation}\label{eq1}
|x_{n+1}| \leq k\lfloor\log_k{{|x|}}\rfloor+2k.
\end{equation}
From Lemma \ref{lemma}, we know that the inequality
\[\lfloor \log_k{r}\rfloor \leq r/k^2 +1,\]
holds for all natural numbers $r$, so we have that \eqref{eq1} implies that
\begin{equation}
|x_{n+1}| \leq \frac{|x_n|}{k}+3k.
\end{equation}
Therefore, for $|x_n| > \frac{3k^2}{k-1}$, we have that
\[|x_{n+1}|<|x_n|,\]
so that the sequence is decreasing in length, which means that for large enough $n$ we have that $|x_n| \leq  \frac{3k^2}{k-1}$, in which case
\[|x_{n+1}| \leq \frac{|x_n|}{k}+3k \leq \frac{3k}{k-1}+3k = \frac{3k^2}{k-1}=:\alpha_k\]
Therefore, for $l = \max\{|x_0|,\lceil \alpha_k\rceil\}$ we have that $|x_n| \leq l$, for all $n \in \mathbb{N}$
\end{proof}
\begin{lemma}\label{lemma}
    For $k\geq 2$ and all $r\in\mathbb{N}$ we have that
    \[\lfloor \log_k{r}\rfloor \leq r/k^2 +1,\]
\end{lemma}
\begin{proof}
    Indeed, notice that since $\lfloor \log_k{r}\rfloor=p$, for $k^p\leq r<k^{p+1}$ it suffices to prove the claim for $r=k^p$, where $p=0,1,2,\dots$. For $p=0,1,2$ the inequality is clear, so we may assume $p\geq 3$. Hence, we need to show that
    \begin{equation*}
        p\leq k^{p-2}+1,
    \end{equation*}
    for $p\geq 3$ and $k\geq 2$. By taking the logarithm on both sides of the inequality, this is equivalent to 
    $$(p-s)\log(k)\geq \log(p-1).$$
    Differentiating the function 
    $$f(x):=\frac{\log(x-1)}{x-2},$$
    we find that it is decreasing for $x\geq 3$. Hence, 
    \begin{equation*}
        \frac{\log(p-1)}{p-2}\leq f(3)=\log(2)\leq \log(k),
    \end{equation*}
    for all $k\geq 2$. (Thank you to Sparkle-Lin for the proof on MathStackExchange).
\end{proof}

This theorem also tells us $|x_n| \leq \lceil\alpha_k\rceil$ for all $n>N \in \mathbb{N}$ sufficiently large.
Thus one would only need to check for words of length less than $\lceil \alpha_k\rceil$ for fixed points in $\Sigma^*_k$. Since there are $k$ letters in $\Sigma_k$, that results in a total of\[ \frac{k^{\lceil \alpha_k\rceil}-1}{k-1}\]
different words, meaning the number of words to be tested grows exponentially with $k$. Theorem \ref{teo2} reduces the number of words to be tested, as follows.
\begin{theorem}\label{teo2}
Let $\overline{x} \in \Sigma_k^*$ be a fixed point of the sequence defined by $x_{n+1} = \mathcal{P}_k(x_n)$, such that $\overline{x}$ is given by the string
\[
\overline{x} = a_{(n_1,1)}a_{(n_1-1,1)}...a_{(0,1)}b_1a_{(n_2,2)}...a_{(0,2)}b_2...a_{(n_r,r)}...a_{(0,r)}b_r,
\]
where $a_{(i,j)} \in \Sigma_k,$ for all $i< n_j,$ and $a_{(n_j,j)} \in \Sigma_k\backslash\{0\}$ and $0\leq n_j$ for $ 1\leq j\leq r.$\\
Then
\begin{equation}
\sum_{j=1}^{r}n_j \geq \sum_{j=1}^{r}k^{n_j}-2r.
\end{equation}
\end{theorem}
\begin{proof}
Initially, note that
\begin{equation}\label{eqbar}
|\overline{x}|=(n_1+1)+1+(n_2+1)+1+...+(n_r+1)+1=\sum_{j=1}^{r}n_j+2r.
\end{equation}
Also, since $\overline{x}$ is a fixed point, $\mathcal{P}_k(\overline{x})=\overline{x}$, therefore
\begin{gather*}
a_{(n_1,1)} \cdot k^{n_1}+...+a_{(0,1)} \cdot k^{0} = |\overline{x}|_{b_1}\\
a_{(n_2,2)} \cdot k^{n_2}+...+a_{(0,2)} \cdot k^{0} = |\overline{x}|_{b_2}\\
\vdots\\
a_{(n_r,r)} \cdot k^{n_r}+...+a_{(0,r)} \cdot k^{0} = |\overline{x}|_{b_r}.
\end{gather*}
Adding all these equalities yields
\begin{equation*}
    \sum_{j=1}^{r}\left(a_{(n_j,j)} \cdot k^{n_j}+...+a_{(0,j)} \cdot k^{0}\right) = |\overline{x}|.
\end{equation*}
Since $0\leq a_{(i,j)}$, we have that
\[
\sum_{j=1}^{r}\left(a_{(n_j,j)} \cdot k^{n_j}\right) \leq |\overline{x}|.
\]
Moreover, since $1\leq a_{(n_j,j)}$ this implies that
\[
\sum_{j=1}^{r}k^{n_j}\leq |\overline{x}|.
\]
Equation \eqref{eqbar} then yields
\[
\sum_{j=1}^{r}k^{n_j}\leq \sum_{j=1}^{r}n_j+2r.
\]
Thus,
\[
\sum_{j=1}^{r}n_j \geq \sum_{j=1}^{r}k^{n_j}\ -2r.
\]
\end{proof}
This result can be generalized as follows.

For $i=1,...,p$, denote by $\overline{x}_i$ the words contained in a $p$-cycle, such that, for $i=1,...,p$ we have that
\[
\overline{x}_i = a_{(n_{(1,i)},1,i)}a_{(n_{(1,i)}-1,1,i)}...a_{(0,1,i)}b_1a_{(n_{(2,i)},2,i)}...a_{(0,2,i)}b_2...a_{(n_{(r,i)},r_i,i)}...a_{(0,r,i)}b_{r}.
\]
\begin{theorem}
Let $\overline{x}_i$ be the words in a $p$-cycle in $\Sigma_k^*$. Then
\[
\sum_{i=1}^{p}\sum_{j=1}^{r}n_{(j,i)} \geq \sum_{i=1}^{p}\sum_{j=1}^{r}k^{n_{(j,i)}}\ -2pr.
\]
\end{theorem}
\begin{proof}
Analogous to the previous proof.
\end{proof}
As a consequence from these results, we can more easily calculate all the fixed points and cycles in a given alphabet.\\
However, it the previous result is hard to implement computation-wise, which motivated the following method:

In $\Sigma_k^*$, define the following equivalence relation:
$$w_1\sim w_2 \iff |w_1|_0=|w_2|_0, |w_1|_1=|w_2|_1, ... , |w_1|_{k-1}=|w_2|_{k-1}.$$
That is, $w_1$ and $w_2$ are equivalent if and only if $w_1$ is an anagram of $w_2$. Notice that $w_1\sim w_2 \implies \mathcal{P}_k(w_1) = \mathcal{P}_k(w_2)$. This is fundamental, since we can consider $\Sigma_k^*/\sim$ to be the set of equivalence classes in $\Sigma_k^*$ and therefore if $w \in [w_0]$, where $w_0$ is a fixed point of $\Sigma_k^*$, then $\mathcal{P}_k(w)=w_0\implies \mathcal{P}_k(\mathcal{P}_k(w))= \mathcal{P}_k(w)$, that is, to check if a class has a fixed point, it's sufficient to take any class representative $w$ and check if $\mathcal{P}_k(\mathcal{P}_k(w))= \mathcal{P}_k(w)$
To check for $p$ cycles, notice that if $w_p$ is in some $p$ cycle, then $\mathcal{P}_k(w_p)$ also is in the same $p$ cycle, therefore if we denote $\mathcal{P}_k\circ\mathcal{P}_k\circ...\circ\mathcal{P}_k$ by $\mathcal{P}_k^p$, we have that, if $w \in [w_p]$, then $\mathcal{P}_k(w) = \mathcal{P}_k^{p+1}(w)$. Obviously if $w_p$ is an $q$ cycle, with $q$ dividing $p$ (that includes $q=1$ i.e.: $w_p$ is a fixed point) then the equality above also holds. Thus to confirm $w_p$ is a $p-cycle$ we would have to check if $\mathcal{P}_k(w) \neq \mathcal{P}_k^{q+1}(w), q|p$\\
This drastically reduces the number of words needed to check to find all fixed points and cycles. Since we know all fixed points and cycles have length less then or equal to $\lceil \alpha_k\rceil$, it's only necessary to check classes of length less then or equal to $\lceil \alpha_k\rceil$.\\
To count how many classes up to length $l$, consider that words of length less than $l$ could be represented by words of length $l$ with leading null letters $\epsilon$ thus the number of ways to arrange the letters $\epsilon,0,1,...,(k-1)$ into words of length $l$ is given by ${k+l}\choose{l}$, though that includes the words $\epsilon\epsilon...\epsilon$ and $\epsilon\epsilon...\epsilon0=0, \epsilon\epsilon...\epsilon00=00,...$,so  removing these there are  ${{k+l}\choose{l}}-l-1$ words. So in, for  $\Sigma_k^*$ the number of classes needed to be checked to find all fixed points of cycles is
$${{k+\lceil \alpha_k\rceil}\choose{\lceil \alpha_k\rceil}}-\lceil \alpha_k\rceil-1 = \frac{(k+\lceil \alpha_k\rceil)!}{k!(\lceil \alpha_k\rceil)!}-\lceil \alpha_k\rceil-1.$$
A simple algorithm to compute a class representative for all classes up to length $l$ in $\Sigma_k^*$, is to order the letters in the following way: $\epsilon < 0<1<...<(k-1)$ and construct words with increasing digits, as follows:
$$\epsilon...\epsilon1,\epsilon...\epsilon2,...,\,\epsilon...\epsilon(k-1), \epsilon...00,\epsilon...01,...,(k-1)(k-1)...(k-1).$$
Note that this covers all possible classes, being necessary to remove only $l-1$ words. Note that although the list includes words with preceding zeros such as $0...001$, whenever there is some letter other than zero, the word is an anagram for a valid word, with no preceding zeros, so it can still be considered a class representative, in the sense that their image through $\mathcal{P}_k$ is the same.
\\
This process still requires immense computational power for base $10$. Luckily, inspired by \cite{e3} we have the following result:
\begin{corollary}
There is no fixed point $w$ in $\Sigma_{k}^*,\, k>4$ with $|w|\geq 2k+2$.
\end{corollary}
\begin{proof}
First, notice that 
\begin{equation}\label{sums}
|w| = \sum_{i=0}^{9}|w|_i=\sum_{|w|_i>0}|(|w|_i)_k|+1.   
\end{equation}
But also note that:\\
If $|w|_i =1 $ then $|(|w|_i)_k|+1=|w|_i+1\implies  |w|_i-|(|w|_i)_k|-1=-1$.\\
If $|w|_i = 2$ then $|(|w|_i)_k|+1=|w|_i\implies|w|_i-|(|w|_i)_k|-1=0$. \\
If $|w|_i \geq 3$ then $|(|w|_i)_k|+1\leq|w|_i-1\implies|w|_i-|(|w|_i)_k|-1\geq1$.\\
So in order for the summations in \eqref{sums} to be equal, for each $i$ such that $|w|_i\geq 3$, there have to be enough other letters $j_i's$ such that $|w|_{j_i}=1$.\\
We claim that if  $|w|_i\geq k$, then
$$|w|_i-|(|w|_i)_k|-1\geq k-3,$$
or equivalently,
$$|(|w|_i)_k|+1\leq  |w|_i-k+3.$$
This is also equivalent to proving that
$$|w|_i -\lfloor\log_k(|w|_i)\rfloor-2\geq k-3,\, \text{ for } |w|_i\geq k,$$
or that
$$|w|_i -\lfloor\log_k(|w|_i)\rfloor-k+1\geq 0,\, \text{ for } |w|_i\geq k.$$
Indeed, let
$$f(x) = x-\lfloor \log_k(x)\rfloor - k+1.$$
First note that
$$f(x)\geq \hat{f}(x) := x-\log_k(x)-k+1.$$
Next, notice that $f(k) = \hat{f}(k) = 0$ and that $$\frac{d}{dx}\hat{f}(x)  = 1-\frac{1}{x\ln(k)}\geq 1-\frac{k}{\ln(k)}>0,$$
for $x\geq k> 4$, so that $\hat{f}$ (and thus also $f$) is increasing and so that the inequality holds for $x\geq k$.
This means that if there exist $i\neq j$, such that $|w|_i,|w|_j\geq k$, then:
$$|w|_i-|(|w|_i)_k|-1\geq k-3, |w|_j-|(|w|_j)_k|-1\geq k-3.$$
Therefore, the sums in \eqref{sums} to be equal, the other terms must satisfy:
$$\sum_{l\neq i,j;|w|_l>0}[|w|_l-|(|w|_i)_k|-1]\leq -(2k-6).$$
But the sum above can only have at most $k-2$ terms (which must all correspond to the case $|w|_i = 1$), so:
$$-2 < -(k-2)\leq\sum_{l\neq i,j;|w|_l>0}[|w|_l-|(|w|_i)_k|-1]\leq -(2k-6)<-2$$ a contradiction. So there can only be one $i$ such that $|w|_i\geq k$.\\
We claim that if $|w|_i\geq k$ for some $i$, then $|w|_i<k^2$.\\
Suppose $|w|_i \geq k^2$ for some $i$. Then $|(|w|_i)_k|\geq 2$. We will show that
$$|w|_i-|(|w|_i)_k|-1\geq k^2-4,\text { for } |w|_i\geq k^2.$$
In fact, this is proved the same way as before, noting that this is equivalent to proving that
$$|w|_i -\lfloor\log_k(|w|_i)\rfloor-2\geq k^2-4,\, \text{ for } |w|_i\geq k.$$
And so it is enough to show that the function
$$g(x) = x - \lfloor\log_k(x)\rfloor -k^2+2$$
satisfies $g(k^2) = 0$ and is increasing.\\
From what we proved, it follows that if there exists $i$ such that $|w|_i\geq k$, then $|(|w_i|)_k| \leq 2$, and in this case it must hold that $|w|_j=1$, for all $j\neq i$. It follows then that $w$ must have at most $(2+1)+(1+1)\cdot(k-1)$ letters. So $|w|\leq 2k+1.$
\end{proof}
In particular, for $k=10$, all fixed points have length less than $22$.\\
Also, generalizing \cite{e3} we have:
\begin{corollary}
There is no fixed point $w\in\Sigma_k^*$, $|w|$ odd and $|w|<2k-2$
\end{corollary}
\begin{proof}
If every letter appears less then $k-1$ times, then $\mathcal{P}_j(w)$ is even, therefore if $|w|$ is odd, there is an odd number of letters $i$ appearing at least $k$ times. If more than 2 letters appear more than $k$ times, then $|w|>2k$, contradiction. So there can be only 1 letter appearing more than $k$ times. If $|w|_i\geq k,\,i>1$ then there is at least $k-1$ occurrences of $i$ counting other letters bigger then 1, and since $w$ is a fixed point that means there are at least $(k-1)\cdot2 = 2k-2$ letters in the word, contradiction. So $i=1$ (since $i=0$ is impossible). If $|w|_1=(10)_k$ or $(11)_k$, then $(k-2)$ occurrences of the letter 1 count other digits, thus $|w|\geq 3+2\cdot(k-2) = 2k-1$ (i.e: $1(k)1(k-1)\dots1(2)10(1)1(0)$), contradiction. If $|w|_1 \geq 12$, then $k$ occurrences of the letter 1 count other digits, which is impossible since there are only $k-1$ other digits, each with at most 1 letter 1 preceding.
\end{proof}

\newpage
\begin{table}[h!]
\centering
 \begin{tabular}{||c | c | c | c | c | c ||} 
 \hline
  \multicolumn{5}{||c||}{Fixed Points for $\Sigma^*_k$, with $k=$} \\
 \hline
 2 & 3 & 4 & 5 & 6 \\  [0.5ex] 
 \hline\hline
 111 & 22 & 22 & 22 & 22 \\
 1001110 &11110&1211110&14233221&14233221\\
  &12111&1311110&14331231&14331231  \\
  &101100&1312111&14333110&14333110 \\
  &1022120&23322110&23322110&15143331 \\ 
  &2211110&33123110&33123110&15233221 \\  
  &22101100&132211110&131211110&15331231 \\ 
  &&&141211110&15333110 \\ 
  &&&141311110&23322110 \\  
  &&&141312111&33123110\\  
  &&&1433223110&1433223110 \\  
  &&&14132211110&1514332231 \\  
  &&&&1533223110\\  
  &&&&14131211110 \\
  &&&&15131211110 \\
  &&&&15141311110 \\  
  &&&&15141312111 \\
  &&&&1514132211110 \\  [1ex]
 \hline
 \end{tabular}
 \caption{Fixed points for the first few values of $k$.}
\end{table}
\begin{table}[h!]
\centering
 \begin{tabular}{||c | c | c | c || } 
 \hline
  \multicolumn{4}{||c||}{Fixed Points ofr $\Sigma^*_k$, with $k=$} \\
 \hline
 7 & 8 & 9 & 10 \\  [0.5ex] 
 \hline\hline
  22& 22 & 22 & 22 \\
  14233221&14233221&14233221&14233221\\
  14331231&14331231&14331231& 14331231 \\
  14333110&14333110&14333110&14333110\\
  15143331&15143331&15143331&15143331 \\ 
  15233221&15233221&15233221&15233221 \\  
  15331231&15331231&15331231&15331231\\ 
  15333110&15333110&15333110&15333110 \\ 
  16143331&16143331&16143331&16143331\\  
  16153331&16153331&16153331&16153331\\  
  16233221&16233221&16233221&16233221\\  
  16331231&16331231&16331231&16331231\\  
  16333110&16333110&16333110&16333110\\  
  23322110&17143331&17143331&17143331 \\
  33123110&17153331&17153331&17153331\\
  1433223110&17163331&17163331&17163331\\  
  1514332231&17233221&17233221&17233221\\
  1533223110&17331231&17331231&17331231\\  
  1614332231&17333110&17333110&17333110\\ 
  1615332231&23322110&18143331&18143331\\ 
  1633223110&33123110&18153331&18153331\\
  &1433223110&18163331&18163331\\
  &1514332231&18173331&18173331\\
  &1533223110&18233221&18233221\\
  &1614332231&18331231&18331231\\
  &1615332231&18333110&18333110\\
  &1633223110&23322110&19143331\\
  &1714332231&33123110&19153331\\
  &1715332231&1433223110&19163331\\
  &1716332231&1514332231&19173331\\
  &1733223110&1533223110&19183331\\
  &16152423324110&1614332231&19233221\\
  &17152423324110&1615332231&19331231\\
  &17161524233241&1633223110&19333110\\
  &17162423324110&1714332231&33123110\\[1ex]

 \hline
 \end{tabular}
\end{table}
\begin{table}[h!]
\centering
 \begin{tabular}{||c| c | c || } 
 \hline
  \multicolumn{3}{||c||}{Fixed Points ofr $\Sigma^*_k$, with $k=$} \\
 \hline
  8&9 & 10 \\  [0.5ex] 
 \hline\hline
  161514131211110&1715332231&1433223110\\
  171514131211110&1716332231&1514332231\\
  171614131211110&1733223110&1533223110\\
 171615131211110&1814332231&1614332231\\
  171615141211110&1815332231&1615332231\\
  171615141311110&1816332231&1633223110\\
  171615141312111&1817332231&1714332231\\ 
  1716251423325110&1833223110&1715332231\\
  17161514132211110&16152423324110&1716332231\\
  &17152423324110&1733223110\\
  &17161524233241&1814332231\\
  &17162423324110&1815332231\\
  &18152423324110&1816332231\\
  &18161524233241&1817332231\\
  &18162423324110&1833223110\\
  &18171524233241&1914332231\\
  &18171624233241&1915332231\\
  &18172423324110&1916332231\\
  &1716251423325110&1917332231\\
  &1816251423325110&1918332231\\
  &1817162514233251&1933223110\\
  &1817162523325110&16152423324110\\
  &1817251423325110&17152423324110\\
  &17161514131211110&17161524233241\\
  &18161514131211110&17162423324110\\
  &18171514131211110&18152423324110\\
  &18171614131211110&18161524233241\\
  &18171615131211110&18162423324110\\
  &18171615141211110&18171524233241\\
  &18171615141311110&18171624233241\\
  &18171615141312111&18172423324110\\
  &181716251433325110&19152423324110\\
  &181726151423326110&19161524233241\\
  &1817161514132211110&19162423324110\\
&	&19171524233241\\
&	&19171624233241\\
&	&19172423324110\\
&	&19181524233241\\
&	&19181624233241\\
&	&19181724233241\\
&	&19182423324110\\
[1ex]
 \hline
 \end{tabular}
\end{table}
\newpage
\begin{table}[h!]
\centering
 \begin{tabular}{|| c || } 
 \hline
  \multicolumn{1}{||c||}{Fixed Points ofr $\Sigma^*_k$, with $k=$} \\
 \hline
 10 \\  [0.5ex] 
 \hline\hline
1716251423325110\\
1816251423325110\\
1817162514233251\\
1817162523325110\\
1817251423325110\\
1916251423325110\\
1917162514233251\\
 1917162523325110\\
 1917251423325110\\
 1918162514233251\\
 1918162523325110\\
 1918171625233251\\
 1918172514233251\\
 1918172523325110\\
 1918251423325110\\
 181726151423326110\\
191726151423326110\\
191817261423326110\\
191817261514233261\\
191817261523326110\\
191826151423326110\\
1817161514131211110\\
1917161514131211110\\
1918161514131211110\\
1918171514131211110\\
1918171614131211110\\
1918171615131211110\\
1918171615141211110\\
1918171615141311110\\
1918171615141312111\\
19182716151423327110\\
191817161514132211110\\
[1ex]
 \hline
 \end{tabular}
\end{table}
\newpage

\begin{table}[h!]
\centering
 \begin{tabular}{||c | c | c | c | c | c ||} 
 \hline
  \multicolumn{5}{||c||}{2-cycles for $\Sigma^*_k$, with $k=$} \\
 \hline
 2 & 3 & 4 & 5 & 6 \\  [0.5ex] 
 \hline\hline
 - & -& - &- & \{152423224110,\,152413423110\} \\ [1ex]
 \hline
 \end{tabular}
 \caption{2 cycles for first few values of $k$}
\end{table}
\begin{table}[h!]
\centering
 \begin{tabular}{||c | c    ||} 

 \hline
  \multicolumn{2}{||c||}{2-cycles for $\Sigma^*_k$, with $k=$} \\
 \hline
 7&8\\  [0.3ex] 
 \hline\hline
  \{152423224110,\,152413423110\}&\{152423224110,\,152413423110\}     \\ 
\{162423224110,\,162413423110\}&\{162423224110,\,162413423110\}\\
\{161524232241,\,161524134231\}&\{172423224110,\,172413423110\}\\
&\{192423224110,\,192413423110\}\\
&\{161524232241,\,161524134231\}\\
&\{171524232241,\,171524134231\}\\
&\{171624232241,\,171624134231\}\\[0.5ex]
 \hline
 \end{tabular}
 \caption{2 cycles for some more values of $k$}
\end{table}
\begin{table}[h!]
\centering
 \begin{tabular}{||c | c    ||} 

 \hline
  \multicolumn{2}{||c||}{Some* 2-cycles for $\Sigma^*_k$, with $k=$} \\
 \hline
 9&10\\  [0.3ex] 
 \hline\hline
  \{152423224110,\,152413423110\}&\{152423224110,\,152413423110\}     \\ 
\{162423224110,\,162413423110\}&\{162423224110,\,162413423110\}\\
\{172423224110,\,172413423110\}&\{172423224110,\,172413423110\}\\
\{182423224110,\,182413423110\}&\{182423224110,\,182413423110\}\\
\{161524232241,\,161524134231\}&\{192423224110,\,192413423110\}\\
\{171524232241,\,171524134231\}&\{161524232241,\,161524134231\}\\
\{181524232241,\,181524134231\}&\{171524232241,\,171524134231\}\\
\{171624232241,\,171624134231\}&\{181524232241,\,181524134231\}\\
\{181624232241,\,181624134231\}&\{191524232241,\,191524134231\}\\
\{181724232241,\,181724134231\}&\{171624232241,\,171624134231\}\\
&\{181624232241,\,181624134231\}\\
&\{191624232241,\,191624134231\}\\
&\{181724232241,\,181724134231\}\\
&\{191724232241,\,191724134231\}\\
&\{191824232241,\,191824134231\}\\[0.5ex]
 \hline
 \end{tabular}
 \caption{2 cycles for more values of $k$, up to length $\lceil\alpha_k\rceil/2$, $\lceil\alpha_k\rceil/3$ respectively}
\end{table}
\clearpage
\begin{table}[h!]
\centering
 \begin{tabular}{||c | c | c | c | c  ||} 
 \hline
  \multicolumn{5}{||c||}{3-cycles for $\Sigma^*_k$, with $k=$} \\
 \hline
 2 & 3 & 4 & 5 & 6 \\  [0.5ex] 
 \hline\hline
 - & \{10210110,\,12111100,\,1212120\}& - &- & - \\ [1ex]
 \hline
 \end{tabular}
 \caption{3 cycles for first few values of $k$}
\end{table}
\begin{table}[h!]
\centering
 \begin{tabular}{||c  ||} 
 \hline
  \multicolumn{1}{||c||}{Some *3-cycles for $\Sigma^*_k$, with $k=$} \\
 \hline7,8,9,10 \\  [0.5ex] 
 \hline\hline
\{16251413424110,\,16153413225110,\,16251423225110\}  \\ [1ex]
 \hline
 \end{tabular}
 \caption{3 cycles for more values of $k$, up to length $\lceil\alpha_k\rceil/2$, $\lceil\alpha_k\rceil/3$ for $k=9,10$ respectively.}
\end{table}


\clearpage
\begin{thebibliography}{99}

\bibitem{Conway98}
Conway, J. H. "The Weird and Wonderful Chemistry of Audioactive Decay." Eureka 46, 5-18, 1986.

\bibitem{OEIS}
OEIS Foundation Inc. (2017), The On-Line Encyclopedia of Integer Sequences, https://oeis.org/A005150

\bibitem{e3}{
Iterative reading of numbers and “black-holes”,Dassow, J., Marcus, S. \& Paun, G. Period Math Hung (1993) 27: 137. https://doi.org/10.1007/BF01876638}

\end{thebibliography}
\end{document}